\documentclass[a4paper,11pt]{amsart}
\usepackage{latexsym}
\usepackage[all]{xy}

\usepackage{amssymb} % \mathbb
\usepackage{amsmath} % \mathbb
\usepackage{amsthm}
\usepackage{thmtools}

\usepackage[style=alphabetic,doi=false,isbn=false,url=false,eprint=true,
backend=biber,giveninits=true,maxcitenames=5,maxbibnames=8,maxalphanames=5,hyperref]{biblatex}
\usepackage{hyperref}
\usepackage{cleveref}
\hypersetup{colorlinks,allcolors=violet}

\usepackage{tikz}
\usetikzlibrary{arrows,cd} 
\usepackage{amsthm}

\def\CC{{\mathbb{C}}}% \C == \mathbb{C}
\def\A{{\mathcal{A}}}% \A == \mathcal{A}
\def\B{{\mathcal{B}}}% \B == \mathcal{B}
\def\O{{\mathcal{O}}}% \G == \mathcal{G}
\DeclareMathOperator{\codim}{codim}

\DeclareMathOperator{\Der}{Der}

\DeclareMathOperator{\pd}{pd}

\DeclareMathOperator{\reg}{reg}

\numberwithin{equation}{section}

\newtheorem{theorem}{Theorem}[section]
\newtheorem{prop}[theorem]{Proposition}
\newtheorem{cor}[theorem]{Corollary}
\newtheorem{lemma}[theorem]{Lemma}
\newtheorem{define}[theorem]{Definition}

\newtheorem{conj}[theorem]{Conjecture}
\theoremstyle{remark}
\newtheorem{rem}[theorem]{Remark}
\newtheorem{example}[theorem]{Example}

\title{Addition theorems for Ziegler pairs of hyperplane arrangements}
\author{Takuro Abe, Lukas K\"uhne and Piotr Pokora}
\date{\today} 
\begin{document}

\begin{abstract}
Inspired by Terao's freeness conjecture, we examine Ziegler pairs, which are pairs of hyperplane arrangements that share the same underlying matroid but have different modules of logarithmic derivations. In this paper, we present a general construction that yields the first known families of Ziegler pairs in arbitrary dimension and size, starting from examples in the complex projective plane.
\end{abstract}
\maketitle
\section{Introduction}

A hyperplane arrangement $\mathcal{A}$ can be studied from two perspectives: a combinatorial one, via its intersection lattice $L(\mathcal{A})$, which is equivalent to a simple matroid; and an algebraic one, via its module of logarithmic derivations $D(\mathcal{A})$. A long-standing open conjecture tries to unify these two points of view.

\begin{conj}[Terao's conjecture \cite{OT}]
	If $\A_1$ and $\A_2$ are two central arrangements defined over the same field with the same combinatorics, that is $L(\A_1)\cong L(\A_2)$, then the derivation module $D(\A_1)$ is free if and only if $D(\A_2)$ is free.
\end{conj}

This conjecture is still wide open, even for line arrangements in $\mathbb{P}^2$.
Note that the assumption on the field is essential as Ziegler observed in~\cite{Ziegler_Characteristic} that there exists a non-free arrangement of $9$ lines over $\mathbb{F}_3$ with underlying matroid the dual Hesse configuration whereas this matroid also admits free realizations over fields of characteristic not $3$.
However, we know that Terao's conjecture (in this stronger form that fixes the field) holds for line arrangements in $\mathbb{P}^2$ over any field with up to $14$ lines~\cite{Tearo14}.

Ziegler was the first to provide evidence against this conjecture, describing a pair of line arrangements with the same combinatorics, but with non-isomorphic modules of derivations~\cite{Ziegler}.
Pairs of arrangements with these properties are therefore nowadays called \emph{Ziegler pairs}, which is the main focus of this paper.
To the best of our knowledge, the known sources of Ziegler pairs are the following.

\begin{itemize}
	\item The Ziegler pair described in~\cite{Ziegler} consisting of two line arrangements with $9$ lines each that intersect in $6$ triple points.
	The degrees of a minimal generating system of their derivation modules are $(1,5,6,6,6)$ and $(1,6,6,6,6,6,6)$.
	It was later pointed out by Schenck that the former case is characterized by the condition that the $6$ triples points lie on a conic~\cite{Sch12}.
	This example was recently studied in detail by Dimca and Sti\-cla\-ru in~\cite{DS24}.
	\item This Ziegler pair was generalized through concepts from \emph{rigidity theory} in~\cite{DST25}. In particular, the authors propose a construction of an infinite family of Ziegler pairs of line arrangements.
	\item It was observed in~\cite{KLP24} that the matroids with a \emph{singular realization space} found in~\cite{CL25} also yield Ziegler pairs.
	The derivation module differs when the realization, a line arrangement with $12$ lines, is in the singular locus of the realization space in these examples.
	\item The \emph{elliptic matroid} $E_n$ is the rank $3$ matroid on the ground set $\{0,1,\ldots,n-1\}$ where a triple is a non-basis if and only if the sum of these numbers is zero modulo $n$.
	The name of this matroid stems from the fact that for $n\ge 10$, the realization space is birational to the elliptic modular curve $X_1(n)$~\cite{BR24}.
	The matroid $E_{11}$ also appeared prominently in~\cite{BHKL25} as a counter-example regarding the space of Lorentzian polynomials.
	Moreover, Cuntz and the third author observed in \cite{CP} that the realization space of $E_{10}$ also yields a Ziegler pair as the derivation module differs on a rational point of this realization space, i.e., the elliptic modular curve $X_1(10)$.
\end{itemize}

Let us mention that in all of these examples there is a \emph{hidden collinearity} in one part of the pair, that is there are collinear intersection points of the arrangement which are not forced to be collinear by the underlying matroid. 

The main result of the paper is a description of the degrees of a minimal system of generators of the derivation of a line arrangement after adding a generic line (\Cref{main}) or after coning the arrangement several times and then adding a generic hyperplane (\Cref{mainhigh}).
This yields new families of Ziegler pairs:
\begin{theorem}[see Corollary \ref{mainT}]
There exist Ziegler pairs of irreducible arrangements of arbitrarily large size and arbitrarily large dimension.
\end{theorem}
In particular, these are the first irreducible (non-trivial) Ziegler pairs in dimensions greater than three.

This article is organized as follows: Section \ref{sec:prelim} presents the necessary preliminaries related to hyperplane arrangements and their derivation modules. Section \ref{sec:generic} focuses on two notions of generic hyperplanes which is the technical heart of our paper. Section \ref{sec:dim3} provides a general recipe for constructing Ziegler pairs in $\mathbb{C}^{3}$ by adding suitably chosen hyperplanes. The final Section~\ref{sec:general} describes a general construction of non-trivial examples of Ziegler pairs in arbitrary dimensions.

\section{Preliminaries}\label{sec:prelim}

We begin with recalling the definitions of the main players of this article. Most of these notions are taken from \cite{OT}.

Let $V = \mathbb{C}^{\ell}$ and consider $S = {\rm Sym}^{*}(V^{*}) \cong \mathbb{C}[x_{1}, \ldots , x_{\ell}]$.
For $\ell=3$, we will assume , starting from Section 4,  that our coordinates are $x, y, z$, and for $\ell=4$, we add a new variable, $w$, to this list.
Define ${\rm Der}(S) =\bigoplus_{i=1}^{\ell} S\cdot \partial_{x_{i}}$. Let $\mathcal{A}$ be a central arrangement of hyperplanes in $V$, i.e., a finite number of linear hyperplanes in $V$.
Throughout this article, for each $H \in \mathcal{A}$ we fix a linear form $\alpha_{H}$ such that ${\rm ker}(\alpha_{H}) = H$, and denote by $Q(\mathcal{A}) = \prod_{H \in \mathcal{A}}\alpha_{H}$ the defining equation of $\mathcal{A}$.
The logarithmic derivation module $D(\mathcal{A})$ of $\mathcal{A}$ is defined as
$$D(\mathcal{A}) = \{\theta \in {\rm Der}(S) \mid \theta(\alpha_{H}) \in S \cdot \alpha_{H} \,\text{ for all } \, H \in \mathcal{A}\}.$$
Recall that $D(\mathcal{A})$ is an $S$-graded reflexive module of rank $\ell$.
We say that the arrangement $\mathcal{A}$ is \emph{free} with the exponents ${\rm exp}(\A) = (d_{1}, \ldots ,d_{\ell})$ if there exists a homogeneous $S$-basis $\theta_{1}, \ldots , \theta_{\ell}$ for $D(\mathcal{A})$ such that ${\rm deg} (\theta_{i}) = d_{i}$ for all $i=1,\dots,\ell$.
For $\theta \in {\rm Der}(S)$, we say that $\theta$ is homogeneous of degree $d$ if 
$\theta(\alpha)$ is zero or a homogeneous polynomial of degree $d$ for all $\alpha \in V^{*}$. 

Observe that if $\mathcal{A}$ is non-empty, every generating set of $D(\mathcal{A})$ contains at least one element of degree $1$, which can always be chosen as the well-known Euler derivation
$$\theta_{E} = \sum_{i=1}^{\ell}x_{i} \partial_{x_{i}}.$$
We define 
$$D_{0}(\mathcal{A}) = \{\theta \in {\rm Der}(S) \mid \theta(Q(\mathcal{A})) = 0 \},$$
and let us recall that there is the following decomposition as a direct sum
$$D(\mathcal{A}) = S\cdot \theta_{E} \oplus D_{0}(\mathcal{A}),$$
hence the freeness of $\mathcal{A}$ boils down to checking the freeness of $D_{0}(\mathcal{A})$.

Given an arrangement $\A$ and a hyperplane $H\in \A$ one often considers the \emph{deletion-restriction-triple} of $H$ in $\A$.
The deletion $\A' := \A\setminus\{H\}$ is $\A$ without $H$.
The restriction $\A^H$ is an arrangement in the hyperplane $H$ defined as
\[
	\A^H := \{K\cap H \mid K\in \A' \}.
\]
This triple yields an exact sequence of derivation modules, the well-studied \emph{Euler sequence}, see~\cite[Prop. 4.45]{OT}:
\[
0\to D(\A')\stackrel{\cdot \alpha_{H}}{\longrightarrow} D(\A) \stackrel{\rho}{\longrightarrow} D(\A^H),
\]
where $\rho$ is the map induced by the quotient map $S\to S/\alpha_{H}S$, more precisely $\rho(\theta)(\overline{f}) = \overline{\theta(f)}$. The map $\rho$ is called the \textit{Euler restriction map}.
This sequence is not right exact in general. However, assuming that $\A'$ is free, the first author has recently proven that the Euler sequence is, in fact, right exact, i.e., $\rho$ is surjective~\cite{Abe24}. A crucial step in our proofs below is to show that this sequence is right exact under certain genericity assumptions. In the course of our proofs we will use the following result. 
\begin{theorem}[{\cite[Theorem~2.5]{ADP}}]
Let $\A$ be an arrangement in $\CC^3$.
\label{ADPTT}
With the above notation, let 
$$
0 \rightarrow \oplus_{i=1}^{m-2} S[-c'_i]
\rightarrow \oplus_{i=1}^m S[-d_i]
\rightarrow D_0(\A') \rightarrow 0
$$
be a minimal free resolution of $D_0(\A')$. If $|\A^{H}|>c_i'-1$ for all $i$, then
the Euler exact sequence 
$$
0 \rightarrow D(\A')[-1]
\stackrel{\cdot \alpha_H}{\longrightarrow} D(\A)
\stackrel{\rho}{\longrightarrow} D(\A^{H}) 
$$
is right exact.
\end{theorem}

Consider a direct sum decomposition $V = V_{1} \oplus V_{2}$ of the vector space $V$. Let $\mathcal{A}_{i}$ be an arrangement in $V_{i}$ and denote by $S_{i}$ be the coordinate ring of $V_{i}$ with $i \in \{1,2\}$. 
Then $S_{1} \otimes_{\mathbb{C}}S_{2} = S$, and we can define
$$\mathcal{A}_{1} \times \mathcal{A}_{2} := \{H_{1} \oplus V_{2} \, : \, H_{1} \in \mathcal{A}_{1}\} \cup \{V_{1} \oplus H_{2} \, : \, H \in \mathcal{A}_{2}\}.$$
We call an arrangement \emph{irreducible} if it cannot be written as a product of two non-empty arrangements.

In this setting, the following result holds.
\begin{prop}[{\cite[Proposition 4.14]{OT}}]
\label{deradd}
$$D(\mathcal{A}_{1} \times \mathcal{A}_{2}) = S\cdot D(A_{1}) \oplus S\cdot D(\mathcal{A}_{2}).$$
\end{prop}
\noindent
Denote by $L(\mathcal{A})$ the intersection lattice of $\mathcal{A}$, i.e.,
$$L(\mathcal{A}) = \{ \cap_{H \in \mathcal{B}}H \, : \, \mathcal{B} \subseteq \mathcal{A}\},$$
which is ordered by reverse inclusion.
Recall that if $\mathcal{A}$ is the product arrangement of the form $\mathcal{A}_{1} \times \mathcal{A}_{2}$, we have a natural isomorphism of the lattices
\begin{equation}
\label{latadd}
L(\mathcal{A}_{1}) \times L(\mathcal{A}_{2}) \cong L(\mathcal{A}_{1} \times \mathcal{A}_{2}).
\end{equation}

A special case of the direct product of an arrangement is its coning which we will use to construct high-dimensional examples.
\begin{define}
	Let $\A$ be an arrangement in $\CC^k$ for some $k\ge 3$ defined by a polynomial $f\in \CC[x_1,\dots,x_k]$.
	The \emph{coning} $c\A$ of $\A$ is the arrangement defined by the polynomial $x_{k+1}\cdot f\in \CC[x_1,\dots,x_{k+1}]$.
\end{define}

The coning of $\A$ is thus the direct product of $\A$ with the $1$-dimensional arrangement consisting of a point on a line.

Let us now define the main object of our studies, namely a Ziegler pair.

\begin{define}
Let $\mathcal{A}_{1}, \mathcal{A}_{2} \subset \mathbb{C}^{\ell}$ be two central hyperplane arrangements such that $L(\mathcal{A}_{1}) \cong L(\mathcal{A}_{2})$. We say that the pair $\mathcal{A}_{1}$ and $\mathcal{A}_{2}$ is a \emph{Ziegler pair} if the associated modules of logarithmic derivations $D(\mathcal{A}_{1})$ and $D(\mathcal{A}_{2})$  have different minimal free resolutions and thus are not isomorphic.
\end{define}
To study the homological properties of hyperplane arrangements, we need the following crucial definition. 

\begin{define}
	Let $\A$ be an arrangement of hyperplanes in $\CC^\ell$. We say that the multiset $(d_1,\ldots,d_s)$ is a \emph{degree sequence} of $\A$, denoted by $\exp(\A)$, if 
	the $0$-th syzygy of a minimal free resolution of $D(\A)$ is of the form 
	$\oplus_{i=1}^s S[-d_i]$.
	We denote the degrees of a minimal generating set of  $D_0(\A)$ by $\exp_0(\A)$.
	\label{degreesequence}
\end{define}

Let $\theta_E,\theta_2,\dots,\theta_s$ be a minimal set of generators of $D(\A)$ for a given arrangement in $\CC^k$.
The module of derivations $D(c\A)$ of the coning is then minimally generated by $\theta_E,\theta_2,\dots,\theta_s,x_{k+1}\partial_{x_k+1}$.
Thus it holds that
\[
	\exp(c\A)=(1,\deg \theta_E,\deg \theta_2,\dots,\deg\theta_s).
\]

Finally, let us recall the following important bound on the Castelnuovo--Mumford regularity of $D(\A)$ that we will use in the forthcoming sections, here for an $S$-graded module $M$ and its minimal free resolution
$$
0 \rightarrow \oplus_{i=1}^{n_s} S[-d_i^s] \rightarrow 
\cdots \rightarrow \oplus_{i=1}^{n_1} S[-d_i^0] \rightarrow M \rightarrow 0,
$$
\textit{the Castelnuovo--Mumford regularity} of $M$ is defined by 
$$
\reg(M):=\max \{ d_{i}^j -j \mid 0 \le j \le s,\ 1 \le i \le n_j\}.
$$
\begin{prop}[{\cite[Proposition 1.2]{MS}}]
	Let $\A$ be an arrangement in~$\CC^\ell$. Then 
	$\reg D(\A) \le |\A|-\ell+1$.
	\label{MS}
\end{prop}
\begin{rem}
The above result for the case of $\ell=3$ was proved earlier by Schenck in \cite[Corollary 3.5]{Sch}.
\end{rem}

\section{Generic hyperplanes}\label{sec:generic}

In this section we collect a number of lemmas about generic hyperplanes.
First, we consider a combinatorial notion of generic hyperplanes and subsequently relate it with a genericity conditions arising from the module of logarithmic derivations.

\begin{define}
	We call a hyperplane $H$ in $\CC^\ell$ \emph{generic with respect to a hyperplane arrangement} $\mathcal{A}$ if for all $0\neq X \in L(\A)$ it holds that 
	\[
		\codim(X\cap H)=\codim(X)+1.
	\]
\end{define}

%\begin{define}
%	We call an arrangement in $\CC^\ell$ \emph{combinatorially generic} if for every subset $\{H_1,\ldots,H_p\}\subseteq\A$ with $p\le \ell$
%	we have $\codim(H_1\cap \dots \cap H_p)=p$ and $H_1\cap\dots\cap H_p=
%    \{0\}$ for $p>\ell$.
%\end{define}
%
%\begin{define}
%A hyperplane $H$ in $\CC^\ell$ is \emph{combinatorially generic with respect to a hyperplane arrangement} $\mathcal{A}$ if the localization $(\A\cup\{H\})_X$ is combinatorially generic for all flats $0\neq X\in L(\mathcal{A}\cup \{H\})$ containing $H$.
%\end{define}

One reason to consider generic hyperplanes stems from the fact that they already impose strong algebraic conditions.

\begin{lemma}\label{lem:comb_generic}
	Let $\A$ be an arrangement in $\CC^3$ and let $\B$ be the arrangement obtained by adding a generic hyperplane $H\notin \A$ with defining equation $\alpha_H$ to $\A$.
	Then it holds that:
	\begin{enumerate}
		\item The Euler sequence
		\[
		0 \rightarrow D(\A) \stackrel{\cdot \alpha_{H}}{\rightarrow }
		D(\B) \stackrel{\rho}{\rightarrow } D(\B^{H})\rightarrow 0.
		\]
		is exact.
		\item  For any minimal set of generators $\theta_E,\theta_2,\dots,\theta_s$ of $D(\A)$ it holds that $\theta_i(\alpha_H)\notin \alpha_H S$ for all $2\le i \le s$.
	\end{enumerate}
\end{lemma}
\begin{proof}
	Set $|\A|=:d$ and let 
	$$
	0 \rightarrow 
	\bigoplus_{j=1}^{n -3} S[-b_j]
	\rightarrow 
	\bigoplus_{i=1}^{n} S[-a_j]
	\rightarrow D(\A) \rightarrow 0
	$$
	be a minimal free resolution of $D(\A)$ with 
	$a_1 \le \ldots \le a_{n},\ 
	b_1 \le \ldots \le b_{n -2}$.
	Note that, $a_{n} \le d-2$ and 
	$b_{n -2} \le d-1$ since the Castelnouvo--Mumford regularity of $D(\A)$ is at most $d-2$, and this follows from \Cref{MS}.
	Since $H$ is generic with respect to $\A$, we know that $|\A^H|=d$, hence $|\A^H|>b_i$ for all $1\le i \le n-2$.
	Thus \Cref{ADPTT} implies that the Euler sequence
		\begin{equation}\label{eq:euler_sequence}
				0 \rightarrow D(\A) \stackrel{\cdot \alpha_{H}}{\rightarrow }
				D(\B) \stackrel{\rho}{\rightarrow } D(\B^{H})=S[-1]
				\oplus S[-d+1] \rightarrow 0.
			\end{equation}
			 is exact which proves (1).

	To prove (2), let $\theta_E=\theta_1,\theta_2,\dots,\theta_{s}$ be a minimal set of generators for $D(\A)$ with $\deg \theta_j=a_j$.
	
	We now show that $\theta_j \not \in D(\B)$ for all $j >1$.
	Assume that $\theta_j \in D(\B)$.
	The above regularity argument implies that $\deg \theta_j=a_j \le d-2$, so after replacing 
	$\theta_j$ by $\theta_j-f \theta_E$, if necessary, it holds that $\rho(\theta_j)=0$ for all $j>1$.
	Thus the exactness of the sequence~\eqref{eq:euler_sequence} implies that $\theta_j/\alpha_{H} \in D(\A)$, contradicting the minimality of the generators of $D(\A)$.
\end{proof}

\begin{lemma}\label{lem:surjectivity}
	Let $\A$ be an irreducible arrangement in $\CC^3$. Let $\A_\ell$ be the arrangement obtained from $\A$ by taking $(\ell-3)$-times the coning operation, i.e., $\A_\ell \subset \CC^\ell$ with new coordinates $x_4,\ldots,x_\ell$.
	Let $H\subset \CC^\ell$ be a generic hyperplane with respect to $\A_\ell$.
	Then for $\B:=\A_\ell \cup \{H\}$ the Euler restriction map
	\[
	\rho \widetilde{D(\B)} \twoheadrightarrow \widetilde{D(\B^H)}
	\]
	is a surjective map on the level of sheaves.
\end{lemma}
\begin{proof}
	Let $p$ be a prime ideal corresponding to some flat $0\neq X\in L(\B)$.
	We need to check that the localized map $\rho_X : D(\B_X)\to D(\B_X^H)$ is surjective.
	
	If $X\not \subset H$, we have $\B_X = (\A_\ell)_X$ which yields a surjective map $\rho_X$.
	Hence we assume $X\subset H$ and thus there exists $Y\in L(\A_\ell)$ such that $Y\cap H = X$.
	This means that $\B_X = (\A_\ell)_Y\cup \{H\}$ which is linearly isomorphic to the coning of  $(\A_\ell)_Y$ as $H$ is generic with respect to  $(\A_\ell)_Y$.
	
	After a change of coordinates, we can assume that
	 \[
	 	D((\A_\ell)_Y) = \langle \theta_1,\dots,\theta_t,\partial_{x_{s+1}},\dots, \partial_{x_\ell}\rangle
	 \]
	 with $\theta_i \in \Der \CC\left[ x_1,\dots,x_s\right]$ for $i=1,\dots,t$ and $H=\{x_{s+1}=0\}$.
	 As  $\B_X$ is the coning of $(\A_\ell)_Y$ with new hyperplane $H$ we get that
	 \begin{align*}
	 	D(\B_X) &= \langle \theta_1,\dots,\theta_t,x_{s+1}\partial_{x_{s+1}},\dots, \partial_{x_\ell}\rangle \\ 
	 	D(\B_X^H)& = \langle \theta_1,\dots,\theta_t,\partial_{x_{s+2}},\dots, \partial_{x_\ell}\rangle.
	 \end{align*}
	 As $\rho_X$ thus maps every generator except $x_{s+1}\partial_{x_{s+1}}$ to its obvious counterpart, this map is surjective.
\end{proof}

We also need an algebraic notion of generic hyperplanes:

\begin{define}
	Let $\mathcal{A}$ be a hyperplane arrangement in $\CC^\ell$ and let $\theta_E$, $\theta_2, \ldots, \theta_s$ be a minimal set of generators of $D(\mathcal{A})$.
	A hyperplane $H$ in $\CC^\ell$ with the defining equation $\alpha_H$ is \emph{algebraically generic} with respect to $\mathcal{A}$ and its generators $\theta_E,\theta_2,\dots,\theta_s$ if $\theta_i(\alpha_H) \notin S\alpha_H$ for all $i=2,\ldots,s$.
\end{define}

One can show that for an irreducible arrangement $\A$ and a fixed derivation $\theta \in D(\A)$ which is not a multiple of the Euler derivation, the set of hyperplanes $\{\alpha_H=0\}$ in the dual space $(\mathbb{P}^{\ell-1})^*$ such that $\theta(\alpha_H)\notin S \alpha_H$ is open and dense.
Since the intersection of finitely many open dense subsets is non-empty, algebraically generic hyperplanes always exist.
Since we do not use this fact below, we omit the proof of this statement.

The notion of an algebraically generic hyperplane might however depend on the chosen generators as the next example shows.

\begin{example}
	Let $\A$ be the deleted $A_3$ arrangement in $\CC^3$ with the defining equation $xyz(x-y)(x-z)=0$.
	One basis of $D(\A)$ is
	\[
	\theta_{1} : =\theta_E, \quad \theta_2:=y(x-y)\partial_y, \quad \theta_3:=z(x-z)\partial_z.
	\]
	The hyperplane $H:\{y-z=0\}$ is algebraically generic with respect to the generators $\theta_2$ and $\theta_3$.
	This hyperplane is however not algebraically generic with respect to the basis $\theta_E, \theta_2, \theta_2+\theta_3$ as $(\theta_2+\theta_3)(y-z)\in S(y-z)$.
\end{example}

In the case of a hyperplane arrangement in rank $3$, a generic hyperplane is also algebraically generic for any minimal set of generators.
\begin{lemma}\label{lem:comb_gen_3}
	Let $\A$ be an essential arrangement in $\CC^3$ and let $H\notin \A$ be a generic hyperplane with respect to $\A$.
	Let $\theta_E,\theta_2,\dots,\theta_s$ be a minimal set of generators of $D(\A)$.
	Then $H$ is also algebraically generic with respect to $\A$ and the generators $\theta_E,\theta_2,\dots,\theta_s$.
\end{lemma}
\begin{proof}
This follows immediately from Lemma~\ref{lem:comb_generic} (2).	
\end{proof}

In the scenario of successive conings, we deduce the following corollary.

\begin{cor}\label{cor:generic}
	Let $\mathcal{A}$ be an arrangement of hyperplanes in $\CC^3$.
	Suppose that $\theta_E,\theta_2,\ldots,\theta_s$ are minimal generators of $D(\A)$.
	For $k=3,\dots,\ell$, let $\A_k$ be the arrangement obtained from $\A$ by taking $(k-3)$-times the coning operation.
	Let $H$ be a generic hyperplane with respect to $\A_\ell$ in $\mathbb{C}^\ell$ and set $H_k:= H \cap \bigcap_{j=k+1}^{\ell} \{x_j=0\}$ for $k=3,\dots,\ell$.
	Then it holds that
		\begin{enumerate}
			\item[(i)] $H_k$ is generic with respect to $\A_k$ for all $k=3,\dots,\ell$.
			\item[(ii)] $H_k$ is algebraically generic with respect to $\A_k$ and its minimal set of generators $\theta_E,\theta_2,\dots,\theta_s,x_4\partial_{x_4},\dots,x_k \partial_{x_k}$ for all $k=3,\dots,\ell$.
		\end{enumerate}
\end{cor}
\begin{proof}
	Let $0\neq X\in L(\A_k)$ be  a flat in $\A_k$.
	By construction, $X$ is then also a flat in $\A_\ell$, hence $\codim(X\cap H)=\codim(X)+1$ in $\CC^\ell$.
	The same equation then also holds in $\CC^k$ and thus $H_{k}$ is also generic with respect to $\A_k$.
	
%	Let $L_1,\dots,L_p$ be hyperplanes in $\A_k$ with $p\ge 1$.
%	By construction, these hyperplanes are also contained in $\A_\ell$ and $H_k=H\cap\bigcap_{j=k+1}^{\ell} \{x_j=0\}$.
%    The first claim follows from the fact that $\codim_{\CC^k}(H_k\cap L_1\cap \dots \cap L_p) $ equals
%	\[
%	\codim_{\CC^\ell}(H\cap \{x_{k+1}=0\}\cap \dots \cap \{x_\ell=0\}\cap L_1\cap \dots \cap L_p).
%	\]

	For the second claim, suppose $H_k$ is defined by the equation $\alpha_{H_k}$.
	Since the derivations $\theta_i$ only involve the variables $x_1,x_2,x_3$ by construction, it follows that $\theta_i(\alpha_{H_k})=\theta_i(\alpha_{H_3})$ for all $i=2,\dots,s$.
	Hence  Lemma~\ref{lem:comb_gen_3} implies that $H_k$ is algebraically generic with respect to $\theta_2,\dots,\theta_s$.
	Lastly note that  $H_k$ is algebraically generic with respect to $x_4\partial_{x_4},\dots,x_k \partial_{x_k}$ since $H_k$ is combinatorially generic with respect to $\A_k$.
\end{proof}

\section{Ziegler pairs in \texorpdfstring{$\mathbb{C}^{3}$}{}}\label{sec:dim3}
Let us present our first main result of the paper.

\begin{theorem}\label{main}
Let \(\mathcal{A}_1,\mathcal{A}_2\) be central arrangements in \(\mathbb C^3\) with
\(L(\mathcal{A}_1)\cong L(\mathcal{A}_2)\) and \(\exp(\mathcal{A}_1)\ne \exp(\mathcal{A}_2)\).
Let \(\mathcal{B}_i=\mathcal{A}_i\cup\{H\}\), where \(H\) is a generic hyperplane, and assume \(L(B_1)\cong L(B_2)\). Then \(B_1,B_2\) form a Ziegler pair.
\end{theorem}

The proof of this result relies on a detailed understanding of the free resolution of the derivation module of the arrangement after adding a generic hyperplane.

\begin{lemma}\label{lem:generic_addition}
	Let $\A$ be an arrangement in $\CC^3$ and let $\B$ be the arrangement obtained by adding a generic hyperplane $H$ to $\A$.
	If $\A$ has exponents $\exp(\A)=(1,a_2,\dots,a_n)$ then the exponents of $\B$ are 
	\[
	\exp(\B)=(1,a_2+1,\dots,a_n+1,|\A|-1).
	\]
\end{lemma}

\begin{proof}
Set $d:=|\A|$. Following Lemma~\ref{lem:comb_generic} (1), there is an exact sequence
\begin{equation}\label{eq:euler_sequence_4}
0 \rightarrow D(\A) \stackrel{\cdot \alpha_{H}}{\rightarrow }
D(\B) \stackrel{\rho}{\rightarrow } D(\B^{H})=S[-1]
\oplus S[-d+1] \rightarrow 0.
\end{equation}

This sequence enables us to determine the degrees of a minimal set of generators for $D(\B).$
Let $\theta_E=\theta_1,\theta_2,\ldots,\theta_{n}$ be a minimal set of generators for $D(\A)$ with $\deg \theta_j=a_j$.
Let $\overline{\theta_E},\varphi$ be a basis for $D(\B^{H})$ with $\deg \varphi=d-1$.
Note that $\overline{\theta_E}=\rho(\theta_E)$.
Since $\rho $ is surjective, there exists $\phi \in D(\B)_{d-1}$ such that 
$\rho(\phi)=\varphi$.

By Lemma~\ref{lem:comb_generic} (2), we obtain that $\theta_j \not \in D(\B)$ for all $j >1$.

Hence, we know that $\theta_E,\alpha_{H} \theta_2,\ldots,\alpha_{H} \theta_{n},\phi \in D(\B)$, and since the sequence~\eqref{eq:euler_sequence_4} is exact we can deduce that these derivations actually generate $D(\B)$.

Let us now show that they form in fact a minimal set of generators for $D(\B)$.
If they generate $D(\B)$ after removing $\phi$, then the image of $\rho$ consists of only $\overline{S} \overline{\theta_E}$, so $\rho$ cannot 
be surjective.
Next, assume that $\alpha \theta_j$ is removable for some $j>1$, where $\alpha:=\alpha_{H}$.
There are $f,f_k,g\in S$ such that
$$
\alpha \theta_{j}=f \theta_E+\sum_{k \neq j} f_k \alpha \theta_k +g \phi.
$$
Since $\deg \alpha\theta_j\le  d-1=\deg \phi$, it holds that $g\in \CC$.
If $g\neq 0$ we could generate $D(\B)$ without $\phi$ which is impossible by the above argument.
Hence it must hold that $g=0$ and thus
$$
\alpha \theta_{j}=f \theta_E+\sum_{k \neq j} f_k \alpha \theta_k.
$$
Since $\alpha \nmid \theta_E$ we must have that $\alpha \mid f$, which shows
$$
\theta_{j}=\frac{f}{\alpha} \theta_E+\sum_{k \neq j} f_k  \theta_k,
$$
 a contradiction to the minimality of the generating set $\{\theta_1, \ldots ,\theta_{n}\}$ for $D(\A)$.
Thus $\{\theta_E,\alpha_{H} \theta_2, \ldots ,\alpha_{H} \theta_{n},\phi\}$ is a minimal set of generators of $D(\B)$ and we obtain $\exp(\B)=(1,a_2+1,\dots,a_n+1,d-1)$ as claimed.
\end{proof}

\begin{proof}[Proof of~\Cref{main}]
\Cref{lem:generic_addition} implies that also set of exponents of $\B_1$ and $\B_2$ is different, hence these two arrangements form a Ziegler pair too.
\end{proof}
\begin{example}
\label{ziegclas}
Let us start with the classical Ziegler example. Consider the two line arrangements 
\begin{multline*}
\mathcal{A}_{1} : f_{1}=xy(x-y-z)(x-y+z)(2x+y-2z)(x+3y-3z)(3x+2y+3z) \\ \times (x+5y+5z)(7x-4y-z) =0 ,
\end{multline*}
\begin{multline*}
\mathcal{A}_{2}: f_{2}=xy(4x-5y-5z)(x-y+z)(16x+13y-20z)(x+3y-3z)\\ \times (3x+2y+3z) (x+5y+5z)(7x-4y-z)=0.
\end{multline*}
Geometrically both arrangements consist of $9$ lines that meet in $6$ triple points.
In $\A_1$ these $6$ triple points lie on a smooth conic and in the case of $\A_2$ they do not lie on a conic.
Thus we know that $L(\mathcal{A}_{1}) \cong L(\mathcal{A}_{2})$.
The exponents of these line arrangements are the following:
$${\rm exp}_{0}(\mathcal{A}_{1}) = (5,6,6,6), {\quad \rm exp}_{0}(\mathcal{A}_{2}) = (6,6,6,6,6,6).$$ Let $H : 13x+171y+232z=0$ and consider now $\mathcal{B}_{i} = \mathcal{A}_{i} \cup H$ with $i \in \{1,2\}$.
Since $H$ is generic with respect to both configurations $\mathcal{A}_{i}$, then $L(\mathcal{B}_{1}) \cong L(\mathcal{B}_{2})$, and by Theorem \ref{main} the exponents of these line arrangements are 
$${\rm exp}_{0}(\mathcal{B}_{1}) = (6,7,7,7,8), \quad {\rm exp}_{0}(\mathcal{B}_{2}) = (7,7,7,7,7,7,8).$$
\end{example}

\begin{example}
Let us consider a Ziegler pair with the property that for each line arrangement the first three exponents are the same. Consider 
\begin{multline*}
\mathcal{A}_{1} : f_{1}=xyz(x+2z)(y+z)(x+2y)(x+5y+3z)(x+y+3z)\\ \times (x+y+z)(3x+5y+3z)=0   
\end{multline*}
and 
\begin{multline*}
\mathcal{A}_{2} : f_{2}=xyz(x+y+z)(y+z)(x+(\sqrt{5}+2)z)(4x+5y+4z)(4x+(\sqrt{5}+3)y) \\ \times (x+y+(\sqrt{5}+3)z)(4x+(5\sqrt{5}+15)y+(4\sqrt{5}+12)z)=0.   
\end{multline*}
We know that the arrangements $\mathcal{A}_{1}$ and $A_{2}$ form a Ziegler pair since they have different minimal free resolutions of the derivation modules, namely
$${\rm exp}_{0}(\mathcal{A}_{1}) = (5,6,6), \quad {\rm exp}_{0}(\mathcal{A}_{2}) = (5,6,6,7),$$
see \cite{CP} for details regarding this example.
Let $H : x+17y+131z = 0$ and define $\mathcal{B}_{i} = \mathcal{A}_{i} \cup H$. Since $H$ is generic with respect to both arrangements $\mathcal{A}_{i}$, then $L(\mathcal{B}_{1}) = L(\mathcal{B}_{2})$, and the arrangements $\mathcal{B}_{1}$, $\mathcal{B}_{2}$ form a Ziegler pair with
$${\rm exp}_{0}(\mathcal{B}_{1}) = (6,7,7,9), \quad {\rm exp}_{0}(\mathcal{B}_{2}) = (6,7,7,8,9).$$
\end{example}
Using the techniques presented in Theorem \ref{main}, we can derive another result that characterizes Ziegler pairs via restrictions.
\begin{theorem}
Let $\A_i=\A_i' \cup \{H\}$ be an arrangement in $\CC^3$ with $i \in \{1,2\}$ and let 
$$
0 \rightarrow F_1^i \rightarrow F_0^i \rightarrow D(\A_i') \rightarrow 0
$$
be minimal free resolutions. Assume that $\A_1',\A_2'$ form a Ziegler pair. If $
F_0^i=\oplus S[-a_j^i],\ 
F_1^i=\oplus S[-b_j^i]$ and 
$\max \{a_j^i\}=a^i,\ \max \{b_j^i\}=b^i$, then $\A_1$ and $\A_2$ form a Ziegler pair if 
$|\A_i^{H}|>\max\{a^1,a^2,b^1,b^2\} $.
\label{genzpC3}
\end{theorem}

\begin{proof}
 The same proof as in Lemma \ref{lem:comb_generic} works since the condition on $|\A_i^{H}|$ forces the vanishing
$$
H^2(\widetilde{D(\A_i')}(-1-b^i+|\A_i^{H}|-1))=0,
$$
which is the key observation in the proof of Lemma \ref{lem:comb_generic}.
\end{proof}

\medskip
\begin{example}
We will explain how to use the above result in the setting of Example \ref{ziegclas}.
Let us compute the minimal free resolutions of the derivation modules of the arrangements $\mathcal{A}_{1}$, $\mathcal{A}_{2}$, namely
$$0 \rightarrow S[-8] \oplus S[-7] \rightarrow S^{3}[-6]\oplus S[-5]\oplus S[-1] \rightarrow D(\mathcal{A}_{1})$$
and
$$0 \rightarrow S^{4}[-7] \rightarrow S^{6}[-6] \oplus S[-1] \rightarrow D(\mathcal{A}_{2}).$$
Since line $H : 13x+171y+232z=0$ is generic with respect to arrangements $\mathcal{A}_{1}$, $\mathcal{A}_{2}$, then $|\A_i^{H}| = 9$ for $i \in \{1,2\}$, and 
$$\max\{a^1,a^2,b^1,b^2\}= \max\{6,8,6,7\} = 8 < |\A_i^{H}|,$$
which confirms that the pair $\mathcal{B}_{i} = \mathcal{A}_{i} \cup \ell$ with $i \in \{1,2\}$ forms a Ziegler pair.
\end{example}

\section{Ziegler pairs in higher dimensions}\label{sec:general}

In this section we construct the first non-trivial Ziegler pairs in $\mathbb{C}^{\ell}$ starting with a Ziegler pair of hyperplane arrangements in $\mathbb{C}^{3}$.
The foundation of our construction is describing the structure of the generators of the derivation module obtained by coning and adding one generic hyperplane.

\begin{theorem}\label{thm:main_general}
	Let $\A$ be an irreducible arrangement in $\CC^3$. Let $\A_\ell$ be the arrangement obtained from $\A$ by taking $(\ell-3)$-times the coning operation, i.e., $\A_\ell \subset \CC^\ell$ with new coordinates $x_4,\ldots,x_\ell$.
	Let $H\subset \CC^\ell$ be a generic hyperplane with respect to $\A_\ell$.
	Then for $\B_\ell:=\A_\ell \cup \{H\}$ and $n:=|\A^{H_3}|-1$ we have
	$$
	\exp_0(\B_\ell)=((2)^{\binom{\ell-2}{2}},(\exp_0(\A)+ 1)^{\ell-2},n).
	$$

    %PP -  not sure about the number of generators, I guess we have (\ell-3)+\binom{\ell-3}{2}} = \binom{\ell-2}{2} generators in degree 2.
%	 for $k\ge 3$ we set $H_k:= H \cap \bigcap_{j=k+1}^{\ell} \{x_j=0\}$, which is a hyperplane in $\CC^k$, and $n+1:=|\A^{H_3}|$.
	\label{mainhigh}
\end{theorem}
\begin{proof}
Let $\theta_E,\theta_2,\ldots,\theta_s$ be a minimal set of generators for $D(\A)$ with $\deg \theta_i=d_i$. Then it is clear that $D(\A_\ell)$ has a minimal set of generators 
$$
\theta_E,\theta_2,\ldots,\theta_s,x_4\partial_{x_4},\ldots,x_\ell\partial_{x_\ell}.
$$

For $k\ge 3$, we define $\A_k$ to be the $(k-3)$--times coned arrangement of~$\A$ and
$\B_k:=\A_k \cup \{H_k\}$.
By the construction of the coned arrangement and the generic hyperplane $H$, we have
\[
	\B_\ell^{H_k}=\B_k^{H_k} = \B_{k-1}.
\]
Let $\rho_k^{k-1}:D(\B_k)\to D(\B_k^{H_k})=D(\B_{k-1})$ be the Euler restriction map, and for $2\le j < i \le \ell$ we denote by $\rho_i^j:D(\B_i)\to D(\B_j)$ the composition of these maps.

\emph{Claim 1:} For every $3\le k \le \ell$, the map $\rho_k^{k-1}:D(\B_k)\to D(\B_{k-1})$ is surjective and therefore the following Euler sequence is exact:
\begin{equation}\label{eq:exact_euler}
0 \rightarrow D(\A_k)[-1] \stackrel{\cdot \alpha_{H_k}}{\rightarrow}
D(\B_k) \stackrel{\rho_k^{k-1}}{\rightarrow } D(\B_{k-1}) \rightarrow 0.
\end{equation}
\emph{Proof of Claim 1:}
The case $k=3$ was proved in~\Cref{main}.
Now assume $k>3$.
Since $\pd_S D(\A_k)\le 1$, this module has a free resolution
$$
0 \rightarrow 
\bigoplus_{j=1}^{p - k} S[-b_j]
\rightarrow 
\bigoplus_{j=1}^{p} S[-a_j]
\rightarrow D(\A_k) \rightarrow 0
$$
for some numbers $a_j, b_j$ and $p=s+k-3$.
Since we assume $k>3$, we know that $H^1(\O(-a_j))=0$ for all $1\le j \le p$ and $H^2(\O(-b_j))=0$ for all $1\le j \le p-k$.
Thus the long exact sequence in cohomology corresponding to the above free resolution shows that $H^1(\widetilde{D(\A_k)(-1)})=0$.

Lemma~\ref{lem:surjectivity} implies that there is an exact sequence of sheaves
\[
0 \rightarrow 
\widetilde{D(\A_k)}(-1) \stackrel{\cdot \alpha_{H_k}}{\rightarrow}
\widetilde{D(\B_k)} \stackrel{\rho_k^{k-1}}{\rightarrow}
\widetilde{D(\B_{k-1})} \rightarrow 0.
\]
The above cohomology-vanishing argument now shows that taking global sections preserves exactness and thus the map $\rho_k^{k-1}$ is surjective, which proves Claim 1.

The module $D(\B_\ell^{H_3})$ is a free module of rank $2$ with a basis $\theta_E,\phi_2$ with $\deg \phi_2 = n$.
Since the map $\rho_k^{k-1}$ is surjective for every $k\ge 3$ we can find successive preimages of $\phi_2$, namely $\phi_k\in D(\B_k)$ such that $\rho_k^2(\phi_k)=\phi_2$ for every $k\ge 3$.

After a change of coordinates, we may assume that $\partial_{x_j}(\alpha_H)=1$ for all $4\le j \le \ell$.
For $2\le i \le s$ and $4\le j\le \ell$, consider the derivation
$$
\varphi_i^j:=x_j \theta_i-\theta_i(\alpha_H) (x_j \partial_{x_{j}}).
$$
Then $\varphi_i^j(\alpha_H)=0$ for all $ 2 \le i \le s,\ 
4 \le j \le \ell$, and $\varphi_i^j \in D(\B_\ell)$ with $\deg \varphi_i^j=d_i+1$.
We also define the derivation for $4 \le i<j$
$$
\eta_i^j:=x_ix_j(\partial_{x_i}-\partial_{x_j}).
$$
Then clearly $\eta_i^j \in D(\B_k)$ for $j \le k$.

For $k\ge 3$, consider the set of derivations
\begin{multline*}
	G_k := \bigg\{\theta_E,\{\alpha_{H_k} x_j \partial_{x_j}\}_{4\le j \le k}, \{\alpha_{H_k} \theta_i\}_{i=2}^s, 
	\{ \varphi_i^j\}_{2 \le i \le s,\ 4\le j \le k}, \\ \{\eta_i^j\}_{4 \le i<j\le k}, \phi_k \bigg\}.
\end{multline*}

\emph{Claim 2:} For each $k\ge 3$, the set $G_k$ generates $D(\B_k)$.

\emph{Proof of Claim 2:} We proceed by induction, where the base case $k=3$ was again proved in~\Cref{main}.
Now assume the claim holds for some $3\le k-1$.

The derivations $\theta_i$ only involve the variables $x_1,x_2,x_3$, and by Corollary~\ref{cor:generic}, it holds that $\theta_i(\alpha_{H_k})\notin S\alpha_{H_k}$ for all $2 \le i \le s$ and $3\le k \le \ell$.
Therefore $\rho_k^{k-1}(\varphi_{i}^j)=\varphi_{i}^j \in D(\B_{k-1})$ for all $4< k, j \le \ell$ and $2 \le i \le s$. Now let us check the images of $\varphi_i^k$ and $\eta_i^k$ under the map $\rho:=\rho_k^{k-1}$.

First, we show that $\rho(\varphi_i^j)=\varphi_i^j$ if $j<k$, and 
$\rho(\varphi_i^k)=-\alpha_{H_{k-1}}\theta_i$.
For $j<k$, there are no $x_k$-terms in $\varphi_i^j$. Thus 
$\rho(\varphi_i^j)=\varphi_i^j$. Let $j=k$, now for $t<k$ 
\begin{eqnarray*}
\rho(\varphi_i^k)(x_t)&=&
x_k \overline{\theta_i(x_t)}\\
&=&-\alpha_{H_{k-1}}\overline{\theta_i(x_t)}=(-\alpha_{H_{k-1}} \theta_i)(x_t)
\end{eqnarray*}
since modulo $x_k$ does not change $\theta_i$. Moreover, 
\begin{eqnarray*}
\rho(\varphi_i^k)(x_k)&=&
-\overline{\theta_i(\alpha_H)} x_k\\
&=&-\alpha_{H_{k-1}}\theta_i(-\alpha_{H_{k-1}}).
\end{eqnarray*}
Since $\overline{x_k}=-\alpha_{H_{k-1}}$, we know that 
$\rho(\varphi_i^k)=-\alpha_{H_{k-1}}\theta_i$.

Next, let us show that $\rho(\eta_i^k)=-\alpha_{H_{k-1}} x_i\partial_{x_i}$.
Clearly $\rho(\eta_i^k)(x_t)=0$ if $t \neq i,k$. We compute 
\begin{eqnarray*}
\rho(\eta_i^k)(x_i)&=&=\overline{x_ix_k}\\
&=&-\alpha_{H_{k-1}} x_i=
(-\alpha_{H_{k-1}} x_i \partial_{x_i})(x_i).
\end{eqnarray*}
Furthermore,
\begin{eqnarray*}
\rho(\eta_i^k)(x_k)&=&\overline{x_i\alpha_{H_{k-1}}}\\
&=&\alpha_{H_{k-1}} x_i=
(-\alpha_{H_{k-1}} x_i \partial_{x_i})(-\alpha_{H_{k-1}}),
\end{eqnarray*}
which means that $\rho(\eta_i^k)=-\alpha_{H_{k-1}} x_i\partial_{x_i}$.
We have also $\rho(\eta_i^j)=\eta_i^j$ for  $4\le i < j <k$.
 Hence, by $\rho_k^{k-1}$, 
$$
\{\varphi_i^j\}_{2 \le i \le s,4 \le j \le k} \mapsto
\{\varphi_i^j,\alpha_{H_{k-1}} \theta_i\}_{2 \le i \le s,4 \le j \le k-1}
$$
and 
$$
\{\eta_i^k\}_{4 \le i \le k-1} \mapsto 
\{\alpha_{H_{k-1}} x_i \partial_{x_i}\}_{4 \le i \le k-1}
$$
up to signs, so we arrive at $\rho_k^{k-1}(G_k)=G_{k-1}$.
Using the exact sequence~\eqref{eq:exact_euler}, we obtain a generating set of $D(\B_k)$ by mapping the generators of $D(\A_k)$ and selecting preimages of the generators of $D(\B_{k-1})$ under the surjective map $\rho_k^{k-1}$.
The latter module is generated by $G_{k-1}$ by our induction assumption.

\emph{Claim 3:} The set $G_k$ is a minimal generating set of $D(\B_k)$. Since the degrees of the set $G_k$ for $k=\ell$ match the claimed exponents, this completes the proof the theorem once this claim is proved.

\emph{Proof of Claim 3:} We again proceed by induction where the base case $k=3$ was proved in~\Cref{main}.
Suppose that $G_{k-1}$ minimally generates $\B_{k-1}$ for some $3\le k-1$.

Since $\rho_k^{k-1}$ is surjective, we cannot remove 
$\theta_E$, $\varphi_i^j$ for $2 \le i \le s, 4 \le j \le k$, $\eta_i^j$ for $4\le i < j \le k$ or $\phi_k$.
Also, since $\A$ is irreducible, the derivations $\theta_2,\dots,\theta_s$ are all of degree at least $2$.
Thus the derivations $\varphi_i^j,\alpha_H \theta_i$  have degree at least $3$.
As $\A$ contains at least $4$ hyperplanes, also the derivation $\phi_k$ has degree at least $3$.

Suppose $x_t\alpha_{H_k} \partial_{x_t}$ is removable, then we have a relation involving the other derivations in $G_k$ of degree $2$, namely
\[
	x_t\alpha_H\partial_{x_t}=
	\sum_{i \neq t} c_ix_i\alpha_{H_k}\partial_{x_i}
	+\sum_{4\le i < j \le k} e_{ij} x_i x_j(\partial_{x_i}-\partial_{x_j}),
\]
where $c_i$ and $e_{ij}$ are constants.
Restricting this equation to $x_k=0$ yields a linear relation among the generators $G_{k-1}$, which is a contradiction to our induction assumption.

Moreover, suppose we can express $\alpha_{H_k}\theta_t$, for $1\le t \le s$, via the relation
\begin{equation}\label{eq:relation}
	\alpha_{H_k}\theta_t=f \theta_E+\sum_{i\neq t} f_i \alpha_{H_k} \theta_i+
	\sum_{\substack{1\le i \le s, \\4\le j \le k}}g_{ij} \varphi_i^j+\sum_{i=4}^k h_{i} x_i\alpha_{H_k} \partial_{x_i}+
	\sum_{4\le i < j \le k} q_{ij}\eta_i^j+
	c \phi_k,
\end{equation}
for polynomials $f, f_i, g_{ij},h_i,q_{ij}\in \mathbb{C}[x_1,\ldots,x_k]$.

The regularity argument from above tells us that for each $i=2,\ldots,s$ the degree of $\theta_i$ is at most $|\A|-2$.
As $\phi_k$ has degree $|\A|-1$, we can conclude that $c$ must be a constant.
Restricting this relation to $H_3$, i.e., $\alpha_{H_k}=\alpha_{H_{k-1}}=\dots = \alpha_{H_{3}}$, yields the relation
\[
0 = \overline{f} \theta_E + \overline{c} \phi_2.
\]
Since these two derivations form a basis of $D(\B_2)$ this yields $c=0$.

%Note that $\varphi_i^j|_{x_\ell=0}=0$ by construction, $\theta_i|_{x_\ell=0}=\theta_i$ and $\alpha_H|_{x_\ell=0}=\alpha_{H_{\ell-1}}$. So restricting the above onto $x_\ell=0$, we have 

%So the rest is to show that they are minimal. 
%Since $\rho_\ell^{\ell-1}$ is surjective, we cannot remove 
%$\theta_E,\varphi_i^j,\eta_i^\ell,\phi$.
%Also, since $\B_\ell$ is irreducible, $\deg \varphi_i^j,\alpha_H \theta_i,\phi$ has degrees at least three. So if $x_i\alpha_H \partial_{x_i}$ is removable, then say $i=t,$
%$$
%x_t\alpha_H\partial_{x_t}=
%\sum_{i \neq t} c_ix_i\alpha_H\partial_{x_i}
%+\sum_{i=4}^{\ell-1} e_i x_\ell x_i(\partial_{x_i}-\partial_{x_\ell}).
%$$
%Here $c_i,e_i$ are all constants. However, by putting $x_\ell=0$, we have a contradiction. So we cannot remove $x_i\alpha_H \partial_{x_i}$.
%Finally, assume that 
%$$
%\alpha_H \theta_1=f \theta_E+\sum_{i=2}^s f_i \alpha_H \theta_i+
%\sum_{i,j}g_{ij} \varphi_i^j+\sum_{i=4}^\ell h_{i} x_i\alpha_H \partial_{x_i}+
%\sum q_{i}\eta_i^\ell+
%c \phi.
%$$
%By constructions, it is clear that, if we restrict this set of generators onto $H_i$, then we obtain the set of minimal generators for $D(\A_i)$ by induction. 
%We may assume that $f=0$ by replacing, say $\phi$ by $\phi+f\theta_E$.
%Assume that $c \neq 0$. By the regularity argument, we know that $\deg \theta_i,\deg \varphi_i^j$ is at most $|\A_3|-1$. Since $\deg \phi=|\A_3|-1$, it holds that $c$ is a constant. Then by restricting the equation onto $H_3$, we know that $0=c\phi$, a contradiction. So $c=0$. 
Note that, by construction, $\varphi_i^k|_{x_k=0}=0$  and $\varphi_i^j|_{x_k=0}=\varphi_i^j$ for $j<k$.
The same conclusion holds for the restriction of $\eta_i^j$.
 Furthermore, it holds that $\theta_i|_{x_k=0}=\theta_i$ and $\alpha_{H_k}|_{x_k=0}=\alpha_{H_{k-1}}$.
 Now restricting the relation~\eqref{eq:relation} onto $x_k=0$ yields
 \begin{align*}
 		\alpha_{H_{k-1}}\theta_t=&\overline{f} \theta_E+\sum_{i\neq t} \overline{f_i} \alpha_{H_{k-1}} \theta_i+\\
 	&\sum_{\substack{1\le i \le s, \\4\le j \le k-1}}g_{ij} \varphi_i^j+\sum_{i=4}^{k-1} \overline{h_{i}} x_i\alpha_{H_{k-1}} \partial_{x_i}+
 	\sum_{4\le i < j \le k-1} \overline{q_{ij}}\eta_i^j.
 \end{align*}
 This is impossible as these are all derivations in $G_{k-1}$ and hence form a minimal generating set of $D(\B_{k-1})$ by our induction hypothesis.
\end{proof}

\begin{cor}
\label{mainT}
	Let $\A_1$ and $\A_2$ be a Ziegler pair in $\CC^3$. Then there exists a Ziegler pair in $\CC^\ell$ determined by $\A_{1}$ and $\A_2$ such that each arrangement in the pair is irreducible.
	\label{highZiegler}
\end{cor}

\begin{proof}
	Since the $\A_i$'s are a Ziegler pairs in $\CC^3$, the hyperplane arrangements obtained by applying the construction in Theorem \ref{mainhigh} have the same combinatorics and different sequences of exponents.
\end{proof}

Concretely, we describe an irreducible Ziegler pair in $\CC^4$ stemming from our construction below.

\begin{example}
	Let us consider these arrangements of hyperplanes in $\mathbb{C}^{4}$:
	\begin{multline*}
		\mathcal{A}_{1} : f_{1} = xy(x-y-z)(x-y+z)(2x+y-2z)(x+3y-3z) \\ \times  (3x+2y+3z)(x+5y+5z)(7x-4y-z)w =0   
	\end{multline*}
	and 
	\begin{multline*}
		\mathcal{A}_{2} : f_{2} = xy(4x-5y-5z)(x-y+z)(16x+13y-20z) \\ \times (x+3y-3z)(3x+2y+3z)(x+5y+5z)(7x-4y-z)w =0.
	\end{multline*}
	These are the conings of the usual Ziegler of $9$ lines and thus we have
	\[
		\exp(\A_1) = (1,5,(6)^3), \quad \exp(\A_2) = (1,(6)^6).
	\]
	Let $H : x+13y+29z+42w = 0$ and consider $\mathcal{B}_{i} = \mathcal{A}_{i} \cup H$ with $i \in \{1,2\}$. We can see that $H$ is generic with respect to both arrangements and thus it holds that $L(\mathcal{B}_{1}) \cong L(\mathcal{B}_{2})$.
	Now Theorem~\ref{thm:main_general} implies ($n=8$ in this case):
	\begin{align*}
	\exp_{0}(\mathcal{B}_{1}) &= (2,(6)^2,(7)^6,8), \\
	\exp_{0}(\mathcal{B}_{2}) &= (2,(7)^{12},8).
	\end{align*}
	Hence we have constructed a non-trivial example of a Ziegler pair of two hyperplane arrangements in $\mathbb{C}^{4}$ since $\mathcal{B}_{1}, \mathcal{B}_{2}$ are irreducible with distinct degree sequences.
\end{example}

\section*{Acknowledgement}
We would like to thank the organizers of the workshop ``New Perspectives on Hyperplane Arrangements'' in May 2025 at Ruhr--Universität Bochum where this project was initiated.
We are moreover grateful to an anonymous referee for several helpful comments on an earlier version of the manuscript.

Takuro Abe is partially supported by JSPS KAKENHI Grant Numbers JP23K17298 and JP25H00399.

Lukas K\"uhne is funded by the Deutsche Forschungsgemeinschaft (DFG, German Research Foundation) – Project-ID 491392403 – TRR 358 and SPP 2458 -- 539866293.

Piotr Pokora is supported by the National Science Centre (Poland) Sonata Bis Grant \textbf{2023/50/E/ST1/00025}. For the purpose of Open Access, the author has applied a CC-BY public copyright license to any Author Accepted Manuscript (AAM) version arising from this submission.

\printbibliography

Takuro Abe,
Department of Mathematics, 
Rikkyo University,
Tokyo 171-8501, Japan.\\
\texttt{Email:} \url{abetaku@rikkyo.ac.jp}
\bigskip

Lukas K\"uhne,
Universit\"at Bielefeld, Fakult\"at f\"ur Mathematik, Bielefeld, Germany.\\
\texttt{Email:} \url{lkuehne@math.uni-bielefeld.de}

\bigskip
Piotr Pokora,
Department of Mathematics,
University of the National Education Commission Krakow,
Podchor\c a\.zych 2,
PL-30-084 Krak\'ow, Poland. \\
\texttt{Email:} \url{piotr.pokora@uken.krakow.pl}
\end{document}